\documentclass[11pt,dvips,twoside]{article}

\usepackage{pslatex}
\usepackage{fancyhdr}
\usepackage{graphicx}
\usepackage{geometry}

\RequirePackage{amsfonts,amssymb,amsmath,amscd,amsthm}
\RequirePackage{txfonts}
\RequirePackage{graphicx}
\RequirePackage{xcolor}
\RequirePackage{geometry}
\RequirePackage{enumerate}

\def\figurename{Figure} 
\makeatletter
\renewcommand{\fnum@figure}[1]{\figurename~\thefigure.}
\makeatother

\def\tablename{Table} 
\makeatletter
\renewcommand{\fnum@table}[1]{\tablename~\thetable.}
\makeatother

\usepackage{amsmath}
\usepackage{amssymb}
\usepackage{amsfonts}
\usepackage{amsthm,amscd}

\newtheorem{theorem}{Theorem}[section]

\newtheorem{corollary}[theorem]{Corollary}
\newtheorem{proposition}[theorem]{Proposition}

\theoremstyle{example}

\theoremstyle{definition}
\newtheorem{definition}[theorem]{Definition}

\theoremstyle{remark}
\newtheorem{remark}[theorem]{Remark}

\numberwithin{equation}{section}

\begin{document}

\title{\bfseries\scshape{Classification, \textbf{$\alpha$-}Inner Derivations and $\alpha$-Centroids of Finite-Dimensional Complex Hom-Trialgebras}}

\author{\bfseries\scshape Bouzid Mosbahi\thanks{e-mail address: mosbahibouzid@yahoo.fr}\\
 University of Sfax, Faculty of Sciences of Sfax, BP 1171, 3000 Sfax, Tunisia.\\
\bfseries\scshape  Ahmed Zahari\thanks{e-mail address: zaharymaths@gmail.com}\\
Universit\'{e} de Haute Alsace,\\
 IRIMAS-D\'{e}partement de Math\'{e}matiques,\\
 18, rue des Fr\`eres Lumi\`ere F-68093 Mulhouse, France.\\
 \bfseries\scshape Imed Basdouri\thanks{e-mail address: basdourimed@yahoo.fr}\\
 University of Gafsa, Faculty of Sciences of Gafsa, 2112 Gafsa, Tunisia.
}

\date{}
\maketitle


\noindent\hrulefill

\noindent {\bf Abstract.}
In the current research work, our basic objective is to investigate the stucture of Hom-associative trialgebras. Next, we build up one important class of  Hom-associative trialgebras
and provide properties of right, left and meddle operations in Hom-associative trialgebras. Furthermore, we describe the classification of $n$-dimensional Hom-associative trialgebras
for $n\leq 3$. Additionally, the properties of the Inner-derivations and centroids are identified and discussed. Eventually the Inner-derivations and centroids are computed.

\noindent \hrulefill

\vspace{.3in}


\vspace{.08in} \noindent \textbf{Keywords}:
Hom-associative trialgebra, Classification, Inner-derivation, Centroid.
\vspace{.3in}
\vspace{.2in}

\pagestyle{fancy} \fancyhead{} \fancyhead[EC]{ }
\fancyhead[EL,OR]{\thepage} \fancyhead[OC]{Bouzid Mosbahi, Ahmed Zahari, Imed Basdouri} \fancyfoot{}
\renewcommand\headrulewidth{0.5pt}

\section{Introduction}
Notably, a Hom-associative trialgebra $(\mathcal{A}, \dashv, \vdash,\bot ,\alpha)$  consists of a vector space, two multiplications and a linear self map. It may be perceived as a deformation of an associative algebra, where the associativity condition is twisted by a linear map $\alpha$, such that when $\alpha=id$, the Hom-associative trialgebras degenerate to exactly triassociative algebras.
We aim in this paper to examine the structure of  Hom-associative trialgebras. Let $\mathcal{A}$ be an $n$-dimensional $\mathbb{K}$-linear space and
$\left\{e_1, e_2, \cdots, e_n\right\}$ be a basis of $\mathcal{A}$. A  Hom-associative trialgebra structure on $\mathcal{A}$ with products $\mu$ and $\lambda$ is determined by $3n^3$
structure constants $\mathcal{\gamma}_{ij}^k$ and $\mathcal{\delta}_{ij}^k$,  where $e_i\dashv e_j=\sum_{k=1}^n\gamma_{ij}^ke_k,\quad e_i\vdash e_j=\sum_{k=1}^n\delta_{ij}^ke_k$ and
$e_i\bot e_j=\sum_{k=1}^n\phi_{ij}^ke_k$  and by $\alpha$  which is given by ${n^2}$ structure constants $a_{ij}$, where $\alpha(e_i)=\sum_{j=1}^na_{ji}e_j$.
Requiring the algebra structure to be a Hom-associative trialgebra  gives rise to sub-variety $\mathcal{H_T}$ of $K^{3n^3+n^2}$. Basic changes in $\mathcal{A}$ yield the natural
transport of structure action of $GL_n(k)$ on $\mathcal{H_T}$. As a matter of facts, isomorphism classes of $n$-dimensional Hom-associative trialgebras are one-to-one correspondence with the
orbits of the action of $GL_n(k)$ on $\mathcal{H_T}.$

Classification problems of the Hom-associative trialgebras using the algebraic  and geometric technique have drawn much interest in the derivations and centroids of
Hom-associative trialgebras. The associative trialgebras introduced by Loday \cite{Ld} with a motivation to provide dual dialgebras, have been further investigated with  regard to several areas
in mathematics and physics. The classification of Hom-associative algebras was set forward elaborated by \cite{AM} and  A. Zahari and I. Bakayoko who addressed the classification of BiHom-associative
dialgebras \cite{AI}.

This paper involves around examining the Inner-derivations and centroids of finite dimensional associative trialgebras. The algebra of Inner-derivations and centroids are highly needed and
extremely useful in terms of algebraic and geometric classification problems of algebras.

The current paper is organized as follows. In the fist section, we identify the topic alongside with some previously obtained results.  The chief objective of this paper is to specify and classify
Inner-derivations as well as centroids of Hom-associative trialgebras. In section 2, we tackle the structure of Hom-associative trialgebras.
In section 3, we handle the algebraic varieties of Hom-associative trialgebras, and we depict classifications, up to isomorphism, of two-dimensional and three-dimensional
Hom-associative trialgebras. We  focus upon the classification of the Inner-derivations. Eventually, in Section 4, we  present the classification of the centroids. In this case,
the concept of derivations and centroids is notably inspired and whetted significant scientific concern from that of finite-dimensional algebras. The algebra of centroids plays an
outstanding role in the classification problems as well as in different applications of algebras. As far as our work is concerned, we elaborated classification results of two and
three-dimensional Hom-associative trialgebras. All considered algebras and vectors spaces  are supposed to be over a field $\mathbb{K}$ of characteristic zero.

\section{Hom-associative trialgebras}
\begin{definition}\label{tia1}
A  Hom-associative trialgebra is a $5$-truple $(\mathcal{A}, \dashv, \vdash, \bot, \alpha)$ consisting of a  linear space $\mathcal{A}$  linear maps
 $\dashv, \vdash, \bot : \mathcal{A}\times \mathcal{A} \longrightarrow \mathcal{A}$ and  $\alpha : \mathcal{A}\longrightarrow \mathcal{A}$ satisfying,
for all $x, y, z\in \mathcal{A}$, the following  conditions :
\begin{eqnarray}
(x\dashv y)\dashv\alpha(z)&=&\alpha(x)\dashv(y\dashv z)\label{eq1},\\
(x\dashv y)\dashv\alpha(z)&=&\alpha(x)\dashv(y\vdash z)=\alpha(x)\dashv(y\bot z),\label{eq1}\\
(x\vdash y)\dashv\alpha(z)&=&\alpha(x)\vdash(y\dashv z),\label{eq12}\\
(x\dashv y)\vdash\alpha(z)&=&\alpha(x)\vdash(y\vdash z)=(x\bot y)\vdash\alpha(z),\label{eq3}\\
(x\vdash y)\vdash\alpha(z)&=&\alpha(x)\vdash(y\vdash z)\label{eq17},\\
(x\bot y)\dashv\alpha(z)&=&\alpha(x)\bot(y\dashv z),\label{eq4}\\
(x\dashv y)\bot\alpha(z)&=&\alpha(x)\bot(y\vdash z),\label{eq5}\\
(x\vdash y)\bot\alpha(z)&=&\alpha(x)\vdash(y\bot z)\label{eq6},\\
(x\bot y)\bot\alpha(z)&=&\alpha(x)\bot(y\bot z)\label{eq7},
\end{eqnarray}
for all $x, y, z\in \mathcal{A}.$
\end{definition}

\begin{remark}\label{rq1}
In addition, $\alpha$ is an endomorphism with respect to $\dashv, \vdash$ and $\bot$. Then,
 $\mathcal{A}$ is said to be a multiplicative Hom-associative trialgebra
\begin{eqnarray}
\alpha(x\dashv y)=\alpha(x)\dashv\alpha(y)&;&\alpha(x\vdash y)=\alpha(x)\vdash\alpha(y),\nonumber\\
\alpha(x\bot y)=\alpha(x)\bot\alpha(y)&,&\nonumber
\end{eqnarray}
for all  $x, y \in \mathcal{A}.$
\end{remark}

\begin{definition}
 A morphism of Hom-associative trialgebra is a linear map
$$f : (\mathcal{A}, \dashv, \vdash,\bot, \alpha)\rightarrow(\mathcal{A}', \dashv',\vdash',\bot', \alpha')$$ such that
$$\alpha'\circ f =f\circ\alpha,$$
and
\begin{eqnarray}
f(x\dashv y)=f(x)\dashv'f(y),\quad f(x\vdash y)=f(x)\vdash'f(y),\quad f(x\bot y)=f(x)\bot'f(y)\nonumber
\end{eqnarray}
 for all  $x, y \in \mathcal{A}.$
\end{definition}

\begin{remark}
A bijective homomorphism is an isomorphism of $\mathcal{A}_1$ and $\mathcal{A}_2$.
\end{remark}
%

\begin{proposition}
Let $(\mathcal{A}, \dashv,\vdash, \bot, \alpha)$ be a Hom-associative trialgebra. Therefore, $(\mathcal{A}, \dashv, \bot, \ast, \alpha)$ is a Hom-associative trialgebra,
where for any $x, y, z\in \mathcal{A},$ $x\ast y=x\vdash y+x\bot y.$
\end{proposition}

\begin{proof}
We prove only one axiom, as others are proven similarly. For any $x, y, z\in \mathcal{A},$
$$\begin{array}{ll}
(x\ast y)\dashv\alpha(z)&=(x\vdash y+x\bot y)\dashv\alpha(z)\\
&=(x\vdash y)\dashv\alpha(z)+(x\bot y)\dashv\alpha(z)\\
&=\alpha(x)\vdash(y\dashv z)+\alpha(x)\bot(y\dashv z)=\alpha(x)\ast(y\dashv z),
 \end{array}$$
which ends the proof.
\end{proof}

\begin{proposition}
Let $(\mathcal{A}, \dashv,\vdash, \bot, \alpha)$ be a Hom-associative trialgebra and  $x\ast y=x\vdash y+x\dashv y+x\bot y.$ From this perspective, $(\mathcal{A},\ast,\alpha)$ is a
Hom-associative algebra.
\end{proposition}

\begin{proof}
For any $x, y, z\in \mathcal{A},$
$$\begin{array}{ll}
(x\ast y)\ast \alpha(z)&-\alpha(x)\ast(y\ast z)\\
&=(x\vdash y+x\dashv y+x\bot y)\ast\alpha(z)-\alpha(x)\ast(y\vdash z+y\dashv z+y\bot z)\\
&=(x\vdash y)\ast\alpha(z)+(x\dashv y)\ast\alpha(z)+(x\bot y)\ast\alpha(z)-\alpha(x)\ast(y\vdash z)-\alpha(x)\ast(y\dashv z)-\alpha(x)\ast(y\bot z)\\
&=(x\vdash y)\vdash\alpha(z)+(x\vdash y)\dashv\alpha(z)+(x\vdash y)\bot\alpha(z)+(x\dashv y)\vdash\alpha(z)+(x\dashv y)\dashv\alpha(z)+(x\dashv y)\bot\alpha(z)\\
&+(x\bot y)\vdash\alpha(z)+(x\bot y)\dashv\alpha(z)+(x\bot y)\bot\alpha(z)-\alpha(x)\vdash(y\vdash z)-\alpha(x)\dashv(y\vdash z)-\alpha(x)\bot(y\vdash z)\\
&-\alpha(x)\vdash(y\dashv z)-\alpha(x)\dashv(y\dashv z)-\alpha(x)\bot(y\dashv z)-\alpha(x)\vdash(y\bot z)-\alpha(x)\dashv(y\bot z)-\alpha(x)\bot(y\bot z).
 \end{array}$$
This corroborates that $(\mathcal{A},\ast,\alpha)$ is a Hom-associative algebra.
\end{proof}

\begin{corollary}
Let $(\mathcal{A}, \dashv, \vdash,\bot,  \alpha)$ be a Hom-associative trialgebra. Then, $\mathcal{A}$ is a Hom-associative algebra with respect to the multiplicative
$\ast : \mathcal{A}\otimes \mathcal{A}\longrightarrow \mathcal{A}$
$$x\ast y=x\dashv y+x\vdash y-x\bot y \quad \text{for any}\quad x, y\in \mathcal{A}.$$
\end{corollary}

\begin{proposition}
Let $(\mathcal{A}, \dashv,\vdash, \bot, \alpha)$ be a Hom-associative trialgebra. Then $(\mathcal{A}, \dashv, \vdash', \alpha)$ is a Hom-associative trialgebra, where
$x \vdash'y=x\vdash y+x\dashv y$, for $x, y\in \mathcal{A}$.
\end{proposition}

\begin{proof}
The proof is straightforward by calculation by using Definition \ref{tia1}.
\end{proof}

\begin{proposition}
Let $(\mathcal{A}, \dashv,\vdash, \bot, \alpha)$ be a Hom-associative trialgebra. Then $(\mathcal{A}, \dashv^{op}, \vdash^{op},\bot^{op}, \alpha)$ is also a Hom-associative trialgebra, with
$x \dashv^{op} y=y\vdash x,\quad x\vdash^{op} y=y\dashv x,\quad x\bot^{op}y=y\bot x$, for any $x, y\in \mathcal{A}$.
\end{proposition}

\begin{proof}
It comes immediately from Definition \ref{tia1}.
\end{proof}

\begin{definition}
A Hom-Leibniz Poisson algebra is the triplet $(A, \cdot, [\cdot , \cdot], \alpha)$
consisting of a  linear space  $\mathcal{T}$, two linear maps  $\cdot, [\cdot, \cdot] : \mathcal{T}\times \mathcal{T} \longrightarrow \mathcal{T}$ and
a linear map  $\alpha : \mathcal{T} \longrightarrow \mathcal{T}$ satisfying the following statements
\begin{enumerate}
\item[i)] $(\mathcal{T}, \cdot, \alpha)$ is a Hom-associative trialbra,
 \item[ii)] $(\mathcal{T}, [\cdot, \cdot], \alpha)$ is a Hom-Leibniz algebra ,
 \item[iii)] $[x \cdot y, \alpha(z)] = \alpha(x) \cdot [y, z] + [x, z] \cdot \alpha(y)$ for all $x, y, z \in \mathcal{T}$
\end{enumerate}
\end{definition}

\begin{proposition}
Let $(\mathcal{A},\dashv,\vdash, \bot, \alpha)$ be a Hom-associative trialgebra. Define now binany operations by
$$\begin{array}{ll}
x\ast y&=x\dashv y-y\vdash x\\
\left[x, y\right]&=x\bot y-y\bot x.
 \end{array}$$
Then $(\mathcal{A}, \ast, \left[\cdot, \cdot\right], \alpha)$ becomes a Hom-associative trialgebra.
\end{proposition}

\begin{proof}
By definition, for any $x, y, z\in \mathcal{A}$, we have
$$\begin{array}{ll}
\left[x, y\right]\ast\alpha(z)&=\left[x, y\right]\dashv\alpha(z)-\alpha(z)\vdash\left[x, y\right]\\
&=(x\vdash y-y\bot x)\dashv \alpha(z)-\alpha(z)\vdash(x\bot y-y\bot y).
 \end{array}$$
and
$$\begin{array}{ll}
\left[x\ast z,\alpha(y)\right]+\left[\alpha(x), y\ast z\right]&=\left[x\dashv z-z\vdash, \alpha(y)\right]+\left[\alpha(x), y\dashv z-z\vdash y\right]\\
&=(x\vdash z-z\vdash x)\bot\alpha(y)-\alpha(y)\bot(x\dashv z-z\vdash x)\\
&+\alpha(x)\bot(y\dashv z-z\vdash y)-(y\dashv z-z\vdash y)\bot\alpha(x).\\
 \end{array}$$
Since
$$\begin{array}{ll}
(x\bot y)\dashv\alpha(z)&=\alpha(x)\bot(y\vdash z),\\
(y\bot x)\dashv\alpha(z)&=\alpha(y)\bot(x\vdash z),\\
\alpha(z)\vdash (x\bot y)&=(z\vdash x)\bot\alpha(y),\\
\alpha(z)\vdash (y\bot x)&=(z\vdash y)\bot\alpha(x),\\
(x\bot z)\bot\alpha(y)&=\alpha(x)\bot(z\vdash y),\\
\alpha(y)\bot(z\vdash x)&=(y\dashv z)\bot\alpha(x),\\
 \end{array}$$
in a  Hom-associative trialgebra, we get
$\left[x, y\right]\ast\alpha(z)=\left[x\ast z, \alpha(y)\right]+\left[\alpha(x), y\ast z\right].$\\
This ends the proof.
\end{proof}

\begin{proposition}
Let $(\mathcal{A}, \dashv, \vdash, \bot, \alpha)$ be a Hom-trialgebra. Therefore, $(\ast, \left[-,-\right], \alpha)$ is a non-commutative Hom-Leibniz-Poisson algebra with respect  to the operations
$x \ast y=x\bot y,\quad \left[x, y\right]=x\dashv y-x\vdash y$, for all $x, y, z\in \mathcal{A}.$
\end{proposition}

\begin{proof}
If following from a straightforward computation.
\end{proof}

\begin{corollary}
Let $(\mathcal{A}, \dashv, \vdash,\bot, \alpha)$ be a Hom-associative trialgebra. Hence, $(\mathcal{A}, \ast, \left[-,-\right], \alpha)$ is a non-commutative Hom-Leibniz-Poisson algebra with
$$x\ast y=x\dashv y+x\vdash y-x\bot y \quad \text{and}\quad\left[x, y\right]=x\ast y-y\ast x\quad \text{for any}\quad x, y\in \mathcal{A}.$$
\end{corollary}

\subsection{Classification of Finite-dimensional Complex Hom-associative trialgebras}
In this section, Hom-associative trialgebras are classified in low dimension.
Note that $\mathcal{TH}_n^m$ denotes $m^{th}$ isomorphism class of Hom-associative trialgebra in dimension $n.$

Let $(\mathcal{A}, \dashv, \vdash,\bot, \alpha)$ be an $n$-dimensionnal Hom-triagebra of $\mathcal{A}$. For any $i, j\in \mathbb{N}, \, 1\leq i,\, j\leq n$, let us put
$$e_i\dashv e_j=\sum_{k=1}^n\gamma_{ij}^ke_k,\quad e_i\vdash e_j=\sum_{k=1}^n\delta_{ij}^ke_k,\quad e_i\bot e_j=\sum_{k=1}^n \phi_{ij}^ke_k,\quad
\alpha(e_i)=\sum_{j=1}^na_{ji}e_j.$$
The axioms in Definition \ref{tia1} are ,respectivly, equivalent to
\begin{eqnarray}
\sum_{p=1}^n\sum_{q=1}^n\gamma_{ij}^pa_{qk}\gamma_{pq}^r&=&\sum_{p=1}^n\sum_{q=1}^na_{pi}\gamma_{jk}^q\gamma_{pq}^r,\\
\sum_{p=1}^n\sum_{q=1}^n\gamma_{ij}^pa_{qk}\gamma_{pq}^r&=&\sum_{p=1}^n\sum_{q=1}^na_{pi}\delta_{jk}^q\gamma_{pq}^r=\sum_{p=1}^n\sum_{q=1}^na_{pi}\phi_{jk}^q\gamma_{pq}^r,\\
\sum_{p=1}^n\sum_{q=1}^n\delta_{ij}^pa_{qk}\gamma_{pq}^r&=&\sum_{p=1}^n\sum_{q=1}^na_{pi}\gamma_{jk}^q\delta_{pq}^r,\\
\sum_{p=1}^n\sum_{q=1}^n\gamma_{ij}^pa_{qk}\delta_{pq}^r&=&\sum_{p=1}^n\sum_{q=1}^na_{pi}\gamma_{jk}^q\delta_{pq}^r=\sum_{p=1}^n\sum_{q=1}^n\phi_{ij}^pa_{qk}\delta_{pq}^r,\\
\sum_{p=1}^n\sum_{q=1}^n\delta_{ij}^pa_{qk}\delta_{pq}^r=\sum_{p=1}^n\sum_{q=1}^na_{pi}\delta_{jk}^q\delta_{pq}^r&,&
\sum_{p=1}^n\sum_{q=1}^n\phi_{ij}^pa_{qk}\gamma_{pq}^r=\sum_{p=1}^n\sum_{q=1}^na_{pi}\gamma_{jk}^q\phi_{pq}^r,\\
\sum_{p=1}^n\sum_{q=1}^n\gamma_{ij}^pa_{qk}\phi_{pq}^r=\sum_{p=1}^n\sum_{q=1}^na_{pi}\delta_{jk}^q\phi_{pq}^r&,&
\sum_{p=1}^n\sum_{q=1}^n\delta_{ij}^pa_{qk}\phi_{pq}^r=\sum_{p=1}^n\sum_{q=1}^na_{pi}\phi_{jk}^q\delta_{pq}^r,\\
\sum_{p=1}^n\sum_{q=1}^n\phi_{ij}^pa_{qk}\phi_{pq}^r&=&\sum_{p=1}^n\sum_{q=1}^na_{pi}\phi_{jk}^q\phi_{pq}^r.
\end{eqnarray}

where ${(a_{ij})}_{1\leq i, j\leq n}$ refers to the matrix of $\alpha$ and $(\gamma_{ij}^k)$, $(\delta_{ij}^k)$ and $(\xi_{ij}^k)$
stand for the structure constants of $\mathcal{A}.$

The axioms in Remark \ref{rq1} are,respectively, equivalent to
\begin{eqnarray}
\sum_{k=1}^n\gamma_{ij}^ka_{qk}=\sum_{k=1}^n\sum_{p=1}^na_{ki}a_{pj}\gamma_{kp}^q,\quad \sum_{k=1}^n\delta_{ij}^ka_{qk}=\sum_{k=1}^n\sum_{p=1}^na_{ki}a_{pj}\delta_{kp}^q,\quad
\sum_{k=1}^n\xi_{ij}^ka_{qk}=\sum_{k=1}^n\sum_{p=1}^na_{ki}a_{pj}\xi_{kp}^q\nonumber.
\end{eqnarray}

\begin{theorem}\label{theo1}
 Any 2-dimensional  complex Hom-associative trialgebra is either  associative or isomorphic to one of the following pairwise non-isomorphic Hom-associative trialgebras:
\end{theorem}

$\mathcal{TH}_2^1$ :	
$\begin{array}{ll}
e_1\dashv e_2=e_1,\\
e_2\dashv e_1=e_1,
\end{array}\quad$
$\begin{array}{ll}
e_2\dashv e_2=e_1,\\
e_1\vdash e_2=e_1,\\
\end{array}\quad$
$\begin{array}{ll}
e_2\vdash e_1=e_1,\\
e_2\vdash e_2=e_1,
\end{array}\quad$
$\begin{array}{ll}
e_1\bot e_2=e_1,\\
e_2\bot e_1=e_1,\\
\end{array}$
$\begin{array}{ll}
e_2\bot e_2=e_1,\\
\alpha(e_2)=e_1.
\end{array}$\\

$\mathcal{TH}_2^2$ :	
$\begin{array}{ll}
e_1\dashv e_2=e_1,\\
e_2\dashv e_1=e_1,\\
\end{array}\quad$
$\begin{array}{ll}
e_2\dashv e_2=e_1,\\
e_1\vdash e_2=e_1,\\
\end{array}\quad$
$\begin{array}{ll}
e_2\vdash e_1=e_1,\\
e_2\vdash e_2=e_1.
\end{array}\quad$
$\begin{array}{ll}
e_2\bot e_2=e_1,\\
\alpha(e_2)=e_1.
\end{array}$\\

$\mathcal{TH}_2^3$ :	
$\begin{array}{ll}
e_1\dashv e_2=e_1,\\
e_2\dashv e_1=e_1,\\
\end{array}\quad$
$\begin{array}{ll}
e_2\dashv e_2=e_1,\\
e_1\vdash e_2=e_1,\\
\end{array}\quad$
$\begin{array}{ll}
e_2\vdash e_1=e_1,\\
e_2\bot e_2=e_1,
\end{array}\quad$
$\begin{array}{ll}
\alpha(e_2)=e_1.
\end{array}$\\

$\mathcal{TH}_2^4$ :	
$\begin{array}{ll}
e_2\dashv e_2=e_1,\\
e_2\vdash e_2=e_1,
\end{array}\quad$
$\begin{array}{ll}
e_2\bot e_2=e_1,\\
\alpha(e_1)=e_1,
\end{array}$
$\begin{array}{ll}
\alpha(e_2)=e_1+e_2.
\end{array}$\\

$\mathcal{TH}_2^5$ :	
$\begin{array}{ll}
e_2\dashv e_1=e_1,\\
e_2\dashv e_2=e_1+e_2,
\end{array}\quad$
$\begin{array}{ll}
e_2\vdash e_1=e_1,\\
e_2\vdash e_2=e_1+e_2,
\end{array}\quad$
$\begin{array}{ll}
e_2\bot e_1=e_1\\
e_2\bot e_2=e_1+e_2,
\end{array}$
$\begin{array}{ll}
\alpha(e_1)=e_1,\\
\alpha(e_2)=e_1+e_2.
\end{array}$\\

$\mathcal{TH}_2^6$ :	
$\begin{array}{ll}
e_1\dashv e_2=e_1,\\
e_2\dashv e_2=e_1+e_2,
\end{array}\quad$
$\begin{array}{ll}
e_1\vdash e_2=e_1,\\
e_2\vdash e_2=e_1+e_2,
\end{array}\quad$
$\begin{array}{ll}
e_1\bot e_2=e_1\\
e_2\bot e_2=e_1+e_2,
\end{array}$
$\begin{array}{ll}
\alpha(e_1)=e_1,\\
\alpha(e_2)=e_1+e_2.
\end{array}$\\

$\mathcal{TH}_2^7$ :	
$\begin{array}{ll}
e_1\dashv e_2=e_1,\\
e_2\dashv e_2=e_2,\\
\end{array}\quad$
$\begin{array}{ll}
e_2\vdash e_1=e_1,\\
e_2\vdash e_2=e_2,
\end{array}\quad$
$\begin{array}{ll}
e_1\bot e_1=e_1\\
e_1\bot e_2=e_1\\
\end{array}\quad$
$\begin{array}{ll}
e_2\bot e_1=e_1\\
e_2\bot e_2=e_2,
\end{array}$
$\begin{array}{ll}
\alpha(e_1)=e_1,\\
\alpha(e_2)=e_2.
\end{array}$\\

$\mathcal{TH}_2^8$ :	
$\begin{array}{ll}
e_1\dashv e_2=e_1,\\
e_2\dashv e_2=e_2,\\
\end{array}\quad$
$\begin{array}{ll}
e_2\vdash e_1=e_1,\\
e_2\vdash e_2=e_2,
\end{array}\quad$
$\begin{array}{ll}
e_1\bot e_2=e_1,\\
e_2\bot e_1=e_1,
\end{array}\quad$
$\begin{array}{ll}
e_2\bot e_2=e_2,
\end{array}$
$\begin{array}{ll}
\alpha(e_1)=e_1,\\
\alpha(e_2)=e_2.
\end{array}$\\

$\mathcal{TH}_2^9$ :	
$\begin{array}{ll}
e_1\dashv e_1=-e_1,\\
e_1\dashv e_2=e_2,\\
e_2\dashv e_1=e_2,
\end{array}\quad$
$\begin{array}{ll}
e_1\vdash e_1=-e_1,\\
e_1\vdash e_2=e_2,\\
e_2\vdash e_1=e_2,
\end{array}\quad$
$\begin{array}{ll}
e_1\bot e_1=-e_1,\\
e_1\bot e_2=e_1,\\
e_2\bot e_2=e_2,
\end{array}\quad$
$\begin{array}{ll}
\alpha(e_1)=e_1,\\
\alpha(e_2)=-e_2.
\end{array}$\\

$\mathcal{TH}_2^{10}$ :	
$\begin{array}{ll}
e_1\dashv e_1=-e_1,\\
e_1\dashv e_2=e_2,\\
e_1\vdash e_1=-e_1,
\end{array}\quad$
$\begin{array}{ll}
e_1\vdash e_2=e_2,\\
e_1\bot e_1=-e_1,\\
e_1\bot e_2=e_2,
\end{array}\quad$
$\begin{array}{ll}
e_2\bot e_1=e_2,\\
\alpha(e_1)=e_1,\\
\alpha(e_2)=-e_2.
\end{array}$\\

$\mathcal{TH}_2^{11}$ :	
$\begin{array}{ll}
e_1\dashv e_1=e_1,\\
e_2\dashv e_2=e_2,
\end{array}\quad$
$\begin{array}{ll}
e_1\vdash e_1=e_1,\\
e_2\vdash e_2=e_2,
\end{array}\quad$
$\begin{array}{ll}
e_1\bot e_1=e_1,\\
e_2\bot e_2=e_2,
\end{array}\quad$
$\begin{array}{ll}
\alpha(e_2)=e_2.
\end{array}$\\

$\mathcal{TH}_2^{12}$ :	
$\begin{array}{ll}
e_2\dashv e_2=e_2,\\
e_2\vdash e_2=e_2,
\end{array}\quad$
$\begin{array}{ll}
e_1\bot e_1=e_1,\\
e_2\bot e_2=e_2,
\end{array}\quad$
$\begin{array}{ll}
\alpha(e_1)=e_1.
\end{array}$\\

$\mathcal{TH}_2^{13}$ :	
$\begin{array}{ll}
e_1\dashv e_2=e_1,\\
e_2\dashv e_2=e_2,\\
e_2\vdash e_1=e_1,
\end{array}\quad$
$\begin{array}{ll}
e_2\dashv e_2=e_2,\\
e_1\bot e_1=e_1,\\
e_1\bot e_2=e_1,
\end{array}\quad$
$\begin{array}{ll}
e_2\bot e_1=e_1,\\
e_2\bot e_2=e_1+e_2,
\end{array}$
$\begin{array}{ll}
\alpha(e_1)=e_1,\\
\alpha(e_2)=e_2.
\end{array}$

\begin{proof}
Let $\mathcal{A}$ be a two-dimensional vector space. To determine a Hom-associative trialgebra structure on $\mathcal{A}$ , we consider $\mathcal{A}$ with respect to one Hom-associative trialgebra operation.
Let $\mathcal{H_{T}}_{2}=(\mathcal{A}, \vdash,\alpha)$ be the Hom-algebra
$$\begin{array}{ll}
e_1 \vdash e_1=-e_1,\quad e_1 \vdash e_2=e_2,\quad e_2 \vdash e_1=e_2,\quad \alpha(e_1)=e_1,\quad \alpha(e_2)=-e_2.
\end{array}$$
The multiplication operations $\dashv,\bot \:in\: \mathcal{A}$ , we define as follows:

$$\begin{array}{ll}
e_1 \dashv e_1= a_1e_1+a_2e_2,\\
e_1 \dashv e_2= a_3e_1+a_4e_2,\\
e_2 \dashv e_1= a_5e_1+a_6e_2,\\
e_2 \dashv e_2= a_7e_1+a_8e_2,
\end{array}\quad
\begin{array}{ll}
e_1 \bot e_1= b_1e_1+b_2e_2,\\
e_1 \bot e_2= b_3e_1+b_4e_2,\\
e_2 \bot e_1= b_5e_1+b_6e_2,\\
e_2 \bot e_2 = b_7e_1+b_8e_2.
\end{array}$$
\noindent Now verifying  Hom-associative trialgebra axioms, we get several constraints for the coefficients $a_i,bi \in \mathbb{R}$ where $1\leq i \leq8.$\\
Applying $(e_1 \dashv e_1)\vdash \alpha(e_1) = \alpha(e_1)\vdash (e_1\vdash e_1)$, we get $(a_1e_1+a_2e_2)\vdash e_1 = e_1\vdash (e_1\vdash e_1)$ and then $e_1 \vdash e_1 = 1$.
Therefore $a_1=-1$.\\
The verification of $(e_1 \vdash e_1) \dashv \alpha(e_1) = \alpha(e_1) \vdash (e_1 \dashv e_1)$
leads to $e_1 \dashv e_1= e_1 \vdash (e_1+a_2e_2)$ and from this we get $e_1+a_2e_2=e_1$. Hence we obtain $a_2=0$.\\
Consider $(e_1 \dashv e_1) \dashv \alpha(e_1) = \alpha(e_1) \dashv (e_1 \vdash e_1)$.
It implies that $e_1 \dashv e_2=1$, therefore $a_3 = 1$ and $a_4 = 0$.\\
The next relation to consider is $(e_1 \dashv e_2) \dashv \alpha(e_1) = \alpha(e_1)\dashv(e_2\dashv e_1)$. It implies that
$1= e_1\dashv(a_5e_1+a_6e_2)$ and we get $a_5 = 0$ and $a_6 = 1$. To find $a_7$ and $a_8$, we note that $(e_2 \dashv e_2) \dashv \alpha(e_1) = \alpha(e_2) \dashv (e_2\vdash e_1)\\
\Rightarrow  (a_7e_1+a_8e_2) \dashv e_1=0 \Rightarrow  a_7e_1+a_6a_8e_2 = 0.$
Hence we have $a_7=0,\,a_6a_8=0$. Finally, we apply $(e_2 \dashv e_2) \dashv \alpha(e_2) = \alpha(e_2) \dashv(e_2\vdash e_2)\Rightarrow a_8(e_2 \dashv e_2) = 0$, and get $a_8 = 0$.\\
Applying $(e_1 \bot e_1)\vdash \beta(e_1) = \alpha(e_1)\vdash (e_1\vdash e_1)$, we get $(y_1e_1+y_2e_2)\vdash e_1 = e_1\dashv (e_1\vdash e_1)$ and then $e_1 \vdash e_1 = 0$.
Therefore $y_1=0 \: and \: y_2=0$. Consider $(e_1 \bot e_2) \dashv \beta(e_1) = \alpha(e_1) \bot (e_2 \dashv e_1)$. It implies that $e_1 \bot e_2=0$, therefore
$y_3 = 0$ and $y_4 = 0$.\\
The next relation to consider is $(e_1 \bot e_2) \dashv \alpha(e_1) = \alpha(e_2)\bot(e_2\dashv e_1)$.\\
It implies that $0= e_2\bot(b_5e_1+b_6e_2)$ and we get $b_5 = 0$ and $b_6 = 0$.
To find $b_7$ and , $b_8$ we note that $(e_2 \bot e_2) \dashv \alpha(e_1) = \alpha(e_2) \dashv (e_2\vdash e_1) \Rightarrow  (b_7e_1+b_8e_2) \dashv e_1=1 \Rightarrow b_7e_1+b_6b_8e_2 = 1.$
Hence we have $b_7=0 ,\, b_6b_8=0$.
Finally, we apply $(e_2 \bot e_2) \dashv \alpha(e_2) = \alpha(e_2) \dashv(e_2\vdash e_2)$
$ \Rightarrow b_8(e_2 \vdash e_2) = 1$, and get $b_8 = 1$. The verification of all other cases leads to the obtained constraints. Thus, in this case we come to the Hom-associative
trialgebra with the multiplication table:
$$\begin{array}{ll}
e_1 \dashv e_1 = -e_1,\\
e_1 \dashv e_2 = e_2,\\
e_2 \dashv e_1 = e_2,\\
\end{array}\qquad
\begin{array}{ll}
e_1 \vdash e_1 = -e_1,\\
e_1 \vdash e_2 = e_2,\\
e_2 \vdash e_1 = e_2,\\
\end{array}\qquad
\begin{array}{ll}
e_1 \bot e_1 = -e_1,\\
e_1 \bot e_2 = e_1,\\
e_2 \bot e_2 = e_2,\\
\end{array}\qquad
\begin{array}{ll}
\alpha(e_1)=e_1,\\
\alpha(e_2)=-e_2.
\end{array}$$
Then $\mathcal{H_{T}}_2=(\mathcal{A}, \dashv,\alpha)$ it is isomorphic to $\mathcal{H_{T}}_2^9$.
The other Hom-associative trialgebras of the list of Theorem \ref{theo1} can be obtained by minor modification of the observation above.
\end{proof}

\begin{theorem}\label{the2}
 Any 3-dimensional  complex Hom-associative trialgebra is either associative or isomorphic to one of the following pairwise non-isomorphic Hom-associative trialgebras:
\end{theorem}

$\mathcal{TH}_3^{1}$ :	
$\begin{array}{ll}
e_2\dashv e_2=e_2+e_3,\\
e_2\dashv e_3=e_2+e_3,\\
e_3\dashv e_2=e_2+e_3,
\end{array}\quad$
$\begin{array}{ll}
e_2\vdash e_2=e_2+e_3,\\
e_3\vdash e_2=e_2+e_3,\\
e_3\vdash e_3=e_2+e_3,
\end{array}\quad$
$\begin{array}{ll}
e_2\bot e_2=e_2+e_3,\\
e_2\bot e_3=e_2+e_3,\\
e_3\bot e_2=e_2+e_3,
\end{array}$
$\begin{array}{ll}
e_3\bot e_3=e_2+e_3,\\
\alpha(e_1)=e_1.
\end{array}$\\

$\mathcal{TH}_3^{2}$ :	
$\begin{array}{ll}
e_2\dashv e_2=e_2+e_3,\\
e_3\dashv e_2=e_2+e_3,\\
e_3\dashv e_3=e_2+e_3,
\end{array}\quad$
$\begin{array}{ll}
e_2\vdash e_2=e_2+e_3,\\
e_2\vdash e_3=e_2+e_3,\\
e_3\vdash e_3=e_2+e_3,
\end{array}\quad$
$\begin{array}{ll}
e_2\bot e_3=e_2+e_3,\\
e_3\bot e_2=e_2+e_3,\\
e_3\bot e_3=e_2+e_3,
\end{array}$
$\begin{array}{ll}
\alpha(e_1)=e_1.
\end{array}$\\

$\mathcal{TH}_3^{3}$ :	
$\begin{array}{ll}
e_1\dashv e_1=ae_2,\\
e_3\dashv e_3=e_3,
\end{array}\quad$
$\begin{array}{ll}
e_1\vdash e_1=e_2,\\
e_3\vdash e_3=e_2,
\end{array}\quad$
$\begin{array}{ll}
e_1\bot e_1=e_2,\\
e_3\bot e_3=e_2,
\end{array}$
$\begin{array}{ll}
\alpha(e_1)=e_1,\\
\alpha(e_2)=e_2.
\end{array}$\\

$\mathcal{TH}_3^{4}$ :	
$\begin{array}{ll}
e_1\dashv e_1=e_3,\\
e_2\dashv e_2=e_2,
\end{array}\quad$
$\begin{array}{ll}
e_1\vdash e_1=e_3,\\
e_2\vdash e_2=e_2,
\end{array}\quad$
$\begin{array}{ll}
e_1\bot e_1=e_3,\\
e_2\bot e_2=e_2,
\end{array}$
$\begin{array}{ll}
\alpha(e_1)=e_1,\\
\alpha(e_3)=e_3.
\end{array}$\\

$\mathcal{TH}_3^{5}$ :	
$\begin{array}{ll}
e_1\dashv e_1=ae_1,\\
e_2\dashv e_2=e_2,\\
e_3\dashv e_2=e_3,
\end{array}\quad$
$\begin{array}{ll}
e_1\vdash e_1=e_1,\\
e_2\vdash e_2=e_2,\\
e_3\vdash e_2=e_3,\\
\end{array}\quad$
$\begin{array}{ll}
e_1\bot e_1=e_1,\\
e_2\bot e_2=e_2,\\
e_3\vdash e_2=e_3,
\end{array}$
$\begin{array}{ll}
\alpha(e_2)=e_2,\\
\alpha(e_3)=e_3.
\end{array}$\\

$\mathcal{TH}_3^{6}$ :	
$\begin{array}{ll}
e_1\dashv e_1=e_1+e_3,\\
e_1\dashv e_3=e_1+e_3,\\
e_3\dashv e_1=e_1+e_3,
\end{array}\quad$
$\begin{array}{ll}
e_1\vdash e_1=e_1+e_3,\\
e_1\vdash e_3=e_1+e_3,\\
e_3\vdash e_1=e_1+e_3,
\end{array}\quad$
$\begin{array}{ll}
e_1\bot e_1=e_1+e_3,\\
e_1\bot e_3=e_1+e_3,\\
e_3\bot e_1=e_1+e_3,
\end{array}$
$\begin{array}{ll}
\alpha(e_2)=e_2.
\end{array}$\\

$\mathcal{TH}_3^{7}$ :	
$\begin{array}{ll}
e_1\dashv e_1=e_1+e_2,\\
e_1\dashv e_2=e_1+e_2,\\
e_2\dashv e_2=e_1+e_2,
\end{array}\quad$
$\begin{array}{ll}
e_1\vdash e_1=e_1+e_2,\\
e_1\vdash e_2=e_1+e_2,\\
e_2\vdash e_1=e_1+e_2,
\end{array}\quad$
$\begin{array}{ll}
e_2\vdash e_2=e_1+e_2,\\
e_2\bot e_1=e_1+e_2,\\
e_2\bot e_2=e_1+e_2,
\end{array}$
$\begin{array}{ll}
\alpha(e_3)=e_3.
\end{array}$\\

$\mathcal{TH}_3^{8}$ :	
$\begin{array}{ll}
e_1\dashv e_2=e_1+e_2,\\
e_2\dashv e_1=e_1+e_2,\\
e_2\dashv e_2=e_1+e_2,
\end{array}\quad$
$\begin{array}{ll}
e_2\vdash e_1=e_1+e_2,\\
e_2\vdash e_2=e_1+e_2,\\
e_1\bot e_1=e_1+e_2,
\end{array}\quad$
$\begin{array}{ll}
e_1\bot e_2=e_1+e_2,\\
e_2\bot e_2=e_1+e_2,
\end{array}$
$\begin{array}{ll}
\alpha(e_3)=e_3.
\end{array}$\\

$\mathcal{TH}_3^{9}$ :	
$\begin{array}{ll}
e_2\dashv e_2=ae_2-be_3,\\
e_2\vdash e_2=e_2+de_3,
\end{array}\quad$
$\begin{array}{ll}
e_2\bot e_2=be_2+e_3,\\
\alpha(e_1)=e_1,\\
\end{array}$
$\begin{array}{ll}
\alpha(e_2)=e_1+e_2,\\
\alpha(e_3)=e_2+e_3.
\end{array}$\\

$\mathcal{TH}_3^{10}$ :	
$\begin{array}{ll}
e_1\dashv e_2=e_1,\\
e_2\dashv e_1=e_1,\\
e_2\dashv e_2=e_1,
\end{array}\quad$
$\begin{array}{ll}
e_2\vdash e_1=e_1,\\
e_2\vdash e_2=e_1,\\
e_2\vdash e_3=e_1,
\end{array}$
$\begin{array}{ll}
e_2\bot e_2=e_1,\\
e_2\bot e_3=e_1,\\
e_3\bot e_2=e_1,\\
\end{array}$
$\begin{array}{ll}
\alpha(e_2)=e_1,\\
\alpha(e_3)=e_3.
\end{array}$\\

$\mathcal{TH}_3^{11}$ :	
$\begin{array}{ll}
e_2\dashv e_1=e_1,\\
e_2\dashv e_2=e_1,\\
e_3\dashv e_2=e_1,\\
\end{array}\quad$
$\begin{array}{ll}
e_1\vdash e_2=e_1,\\
e_2\vdash e_1=e_1,\\
e_3\vdash e_2=e_1,
\end{array}$
$\begin{array}{ll}
e_2\bot e_2=e_1,\\
e_3\bot e_2=e_1,
\end{array}$
$\begin{array}{ll}
\alpha(e_2)=e_1,\\
\alpha(e_3)=e_3.
\end{array}$\\

$\mathcal{TH}_3^{12}$ :	
$\begin{array}{ll}
e_2\dashv e_1=e_1+e_3,\\
e_2\dashv e_2=e_1+e_3,\\
e_3\dashv e_3=e_1+e_3,\\
\end{array}\quad$
$\begin{array}{ll}
e_1\dashv e_2=e_1+e_3,\\
e_2\vdash e_1=e_1+e_3,\\
e_2\vdash e_3=e_1+e_3,
\end{array}$
$\begin{array}{ll}
e_2\bot e_2=e_1+e_3,\\
e_2\bot e_3=e_1+e_3,\\
e_3\bot e_3=e_1+e_3,
\end{array}$
$\begin{array}{ll}
\alpha(e_2)=e_1.
\end{array}$\\

$\mathcal{TH}_3^{13}$ :	
$\begin{array}{ll}
e_3\dashv e_2=e_1+e_3,\\
e_3\dashv e_3=e_1+e_3,\\
e_1\vdash e_2=e_1+e_3,
\end{array}\quad$
$\begin{array}{ll}
e_2\vdash e_3=e_1+e_3,\\
e_3\vdash e_2=e_1+e_3,\\
e_3\vdash e_3=e_1+e_3,
\end{array}$
$\begin{array}{ll}
e_2\bot e_2=e_1+e_3,\\
e_3\bot e_2=e_1+e_3,\\
e_3\bot e_3=e_1+e_3,
\end{array}$
$\begin{array}{ll}
\alpha(e_2)=e_1.
\end{array}$\\

$\mathcal{TH}_3^{14}$ :	
$\begin{array}{ll}
e_2\dashv e_3=e_1+e_3,\\
e_3\dashv e_2=e_1+e_3,\\
e_3\dashv e_3=e_1+e_3,
\end{array}\quad$
$\begin{array}{ll}
e_2\vdash e_2=e_1+e_3,\\
e_3\vdash e_2=e_1+e_3,\\
e_3\vdash e_3=e_1+e_3,
\end{array}$
$\begin{array}{ll}
e_3\bot e_2=e_1+e_3,\\
e_3\bot e_3=e_1+e_3,
\end{array}$
$\begin{array}{ll}
\alpha(e_2)=e_1.
\end{array}$\\

$\mathcal{TH}_3^{15}$ :	
$\begin{array}{ll}
e_1\dashv e_2=e_1+e_3,\\
e_3\dashv e_2=e_1+e_3,\\
e_3\dashv e_3=e_1+e_3,
\end{array}\quad$
$\begin{array}{ll}
e_2\vdash e_3=e_1+e_3,\\
e_3\vdash e_2=e_1+e_3,\\
e_3\vdash e_3=e_1+e_3,
\end{array}$
$\begin{array}{ll}
e_2\bot e_1=e_3,\\
e_2\bot e_3=e_3,\\
e_3\bot e_3=e_3,
\end{array}$
$\begin{array}{ll}
\alpha(e_2)=e_1.
\end{array}$\\

$\mathcal{TH}_3^{16}$ :	
$\begin{array}{ll}
e_2\dashv e_1=e_3,\\
e_2\dashv e_2=e_3,\\
e_3\dashv e_3=e_1+e_3,
\end{array}\quad$
$\begin{array}{ll}
e_1\vdash e_2=e_1,\\
e_3\vdash e_2=e_1+e_3,\\
e_3\vdash e_3=e_1,
\end{array}$
$\begin{array}{ll}
e_3\bot e_2=e_3,\\
e_3\bot e_3=e_3,
\end{array}$
$\begin{array}{ll}
\alpha(e_2)=e_1.
\end{array}$\\

$\mathcal{TH}_3^{17}$ :	
$\begin{array}{ll}
e_1\dashv e_2=e_1+e_3,\\
e_2\dashv e_1=e_1+e_3,\\
e_1\vdash e_2=e_1,
\end{array}\quad$
$\begin{array}{ll}
e_3\vdash e_3=e_1,\\
e_1\bot e_2=e_3,\\
e_2\bot e_1=e_3,
\end{array}$
$\begin{array}{ll}
e_3\bot e_2=e_1,\\
e_3\bot e_3=e_3,
\end{array}$
$\begin{array}{ll}
\alpha(e_2)=e_1.
\end{array}$\\

$\mathcal{TH}_3^{18}$ :	
$\begin{array}{ll}
e_2\dashv e_3=e_1+e_3,\\
e_3\dashv e_2=e_1+e_3,\\
e_3\dashv e_3=e_1+e_3,
\end{array}\quad$
$\begin{array}{ll}
e_2\vdash e_1=e_3,\\
e_2\vdash e_2=e_1+e_3,\\
e_3\vdash e_3=e_1+e_3,
\end{array}$
$\begin{array}{ll}
e_1\bot e_2=e_1+e_3,\\
e_2\bot e_3=e_1,\\
e_3\bot e_2=e_1+e_3,
\end{array}$
$\begin{array}{ll}
\alpha(e_2)=e_1.
\end{array}$\\

$\mathcal{TH}_3^{19}$ :	
$\begin{array}{ll}
e_2\dashv e_3=e_2,\\
e_3\dashv e_2=e_2,\\
e_3\dashv e_3=e_3,
\end{array}\quad$
$\begin{array}{ll}
e_2\vdash e_2=e_2+e_3,\\
e_3\vdash e_3=e_2+e_3,\\
e_2\bot e_2=e_2+e_3,
\end{array}\quad$
$\begin{array}{ll}
e_2\bot e_3=e_3,\\
e_3\bot e_2=e_3,\\
e_3\bot e_3=e_3,
\end{array}$
$\begin{array}{ll}
\alpha(e_1)=e_1.
\end{array}$\\

$\mathcal{TH}_3^{20}$ :	
$\begin{array}{ll}
e_1\dashv e_3=e_1,\\
e_2\dashv e_3=e_1,\\
e_3\dashv e_3=e_1,
\end{array}\quad$
$\begin{array}{ll}
e_3\vdash e_1=e_1,\\
e_3\vdash e_2=e_1,\\
e_3\vdash e_3=e_1,
\end{array}$
$\begin{array}{ll}
e_1\bot e_3=e_1,\\
e_2\bot e_3=e_1,\\
e_3\bot e_3=e_1,
\end{array}$
$\begin{array}{ll}
\alpha(e_2)=e_1,\\
\alpha(e_3)=e_2.
\end{array}$\\

$\mathcal{TH}_3^{21}$ :	
$\begin{array}{ll}
e_2\dashv e_3=e_1,\\
e_3\dashv e_2=e_1,\\
e_3\dashv e_3=e_1,
\end{array}\quad$
$\begin{array}{ll}
e_3\vdash e_1=e_1,\\
e_3\vdash e_3=e_1,\\
e_1\bot e_3=e_1,
\end{array}$
$\begin{array}{ll}
e_2\bot e_3=e_1,\\
e_3\bot e_1=e_1,\\
e_3\bot e_2=e_1,
\end{array}$
$\begin{array}{ll}
\alpha(e_2)=e_1,\\
\alpha(e_3)=e_2.
\end{array}$

\begin{proof}
 Let $\mathcal{A}$ be a three-dimensional vector space. To determine a Hom-associative trialgebra
structure on $\mathcal{A}$ , we consider $\mathcal{A}$ with respect to one Hom-associative trialgebra operation. Let
$\mathcal{H_{T}}_3=(\mathcal{A}, \dashv,\alpha)$ be the Hom-algebra
$$\begin{array}{ll}
e_3 \vdash e_1=e_1,\quad e_3 \vdash e_2=e_1,\quad e_3 \vdash e_3=e_1,\quad \alpha(e_1)=e_1,\quad \alpha(e_3)=e_2.
\end{array}$$
The multiplication operations $\vdash,\bot in \mathcal{A}$. We use the same method of the Proof of the Theorem \ref{the2}.
Then $\mathcal{H_{T}}_3=(\mathcal{A}, \dashv,\alpha)$ it is isomorphic to $\mathcal{H_{T}}_3^{20}$.
The other Hom-associative trialgebra of the list of Theorem \ref{the2} can be obtained by minor modification of the observation.
\end{proof}

\section{$\alpha$-\textbf{Inner-derivations} of Finite-dimensional Hom-associative trialgebras.}
This section sets forward a detailed description of $\alpha$-\textbf{Inner-derivations} of Hom-associative trialgebras in dimensions two and three over the field $\mathbb{K}.$

\begin{definition}\label{dia2}
An $\alpha$-\textbf{Inner-derivation} of the Hom-associative trialgebra $\mathcal{A}$ is a linear transformation $\mathcal{I} : \mathcal{A} \rightarrow \mathcal{A}$ satisfying
\begin{eqnarray}
\alpha\circ\mathcal{I}=\mathcal{I}\circ\alpha&,&\\
ad_z(x)&=&\mathcal{I}(x)\dashv\alpha(z)-\alpha(x)\dashv\mathcal{I}(z),\\
ad_z(x)&=&\mathcal{I}(x)\vdash\alpha(z)-\alpha(x)\vdash\mathcal{I}(z),\\
ad_z(x)&=&\mathcal{I}(x)\bot\alpha(z)-\alpha(x)\bot\mathcal{I}(z),
\end{eqnarray}
for all $x, y\in \mathcal{A}.$
 \end{definition}
Let $(\mathcal{A},\dashv,  \vdash, \alpha) $ be a Hom-associative trialgebra over $\mathbb{K}.$ For $z\in \mathcal{A}$, we have
$$
ad_z(X)=X\dashv z-z\vdash X, \quad (\forall, X\in \mathcal{A}).
$$
We can,therefore, prove that $ad_z$ is a $\alpha$-derivation of $(\mathcal{A},\dashv,  \vdash, \alpha)$. For any $X, Y\in \mathcal{A}$, we get
$$ad_z(X\dashv Y)=(X\dashv Y)\dashv\alpha(z)-\alpha(z)\vdash (X\dashv Y)$$ and
\begin{eqnarray}
ad_z(X)\dashv\alpha(Y)+\alpha(X)\dashv ad_z(Y)&=&(X\dashv z-z\vdash X)\dashv\alpha(Y)+\alpha(X)\dashv(Y\dashv z-z\vdash Y)\nonumber\\
&=&\alpha(X)\dashv(z\dashv Y)-(z\vdash X)\dashv\alpha(Y)+\alpha(X)\dashv(Y\dashv z)-\alpha(X)\dashv(z\dashv Y)\nonumber\\
&=&\alpha(X)\dashv(Y\dashv z)-(z\vdash X)\dashv\alpha(Y)=\alpha(X)\dashv(Y\bot z)-(z\vdash X)\dashv\alpha(Y)\nonumber.
\end{eqnarray}
Hence,
$$
ad_z(X\dashv Y)=ad_z(X)\dashv\alpha(Y)+\alpha(X)\dashv ad_z(Y).
$$
On the other side,
\begin{eqnarray}
ad_z(X\vdash Y)=(X\vdash Y)\dashv\alpha(z)+\alpha(z)\vdash(X\vdash Y)=(X\vdash Y)\dashv(z)-\alpha(z)\dashv(X\bot Y)\nonumber
\end{eqnarray}
and

\begin{eqnarray}
ad_z(X)\vdash\alpha(Y)+\alpha(X)\vdash ad_z(Y)&=&(X\dashv z-z\vdash X)\vdash\alpha(Y)+\alpha(X)\vdash(Y\dashv z-z\vdash Y)\nonumber\\
&=&(X\vdash z)\vdash\alpha(Y)-(z\vdash X)\vdash\alpha(Y)+\alpha(X)\vdash(Y\dashv z)-(X\vdash z)\vdash\alpha(Y)\nonumber\\
&=&\alpha(X)\vdash(Y\dashv z)-(z\vdash X)\vdash\alpha(Y)\nonumber.
\end{eqnarray}
Thus, it follows that
\begin{eqnarray}
ad_z(X\vdash Y)=ad_z(X)\vdash\alpha(Y)+\alpha(X)\vdash ad_z(Y).\nonumber
\end{eqnarray}

 Let $\left\{e_1,e_2, e_3,\cdots, e_n\right\}$ be a basis of an $n$-dimensional Hom-associative trialgebra $\mathcal{A}.$ The product of basis is denoted by
\begin{eqnarray}
\mathcal{I}(e_p)=\sum_{q=1}^n\mathcal{I}_{qp}e_q\nonumber.
\end{eqnarray}
We have
\begin{eqnarray}\label{Ieq2}
\sum_{p=1}^n\mathcal{I}_{pk}a_{qp}=\sum_{p=1}^na_{pk}\mathcal{I}_{qp}& , &\label{deq1}\label{Ieq2}\\
\sum_{k=1}^n\gamma_{ij}^p\mathcal{I}_{rp}=\sum_{p=1}^n\sum_{q=1}\mathcal{I}_{pi}a_{qj}\gamma_{pq}^r&-&\sum_{p=1}^n\sum_{q=1}a_{pi}\mathcal{I}_{qj}\gamma_{pq}^r\label{deq2},\label{Ieq3}\\
\sum_{k=1}^n\delta_{ij}^p\mathcal{I}_{rp}=\sum_{p=1}^n\sum_{q=1}\mathcal{I}_{pi}a_{qj}\delta_{pq}^r&-&\sum_{p=1}^n\sum_{q=1}a_{pi}\mathcal{I}_{qj}\delta_{pq}^r\label{deq3},\label{Ieq4}\\
\sum_{k=1}^n\phi_{ij}^p\mathcal{I}_{rp}=\sum_{p=1}^n\sum_{q=1}\mathcal{I}_{pi}a_{qj}\phi_{pq}^r&-&\sum_{p=1}^n\sum_{q=1}a_{pi}\mathcal{I}_{qj}\phi_{pq}^r\label{deq4}\label{Ieq5}.
\end{eqnarray}

\begin{theorem}\label{Ithieo1}
The  $\alpha$-\textbf{Inner-derivations} of $2$-dimensional Hom-associative trialgebras have the following form :
\end{theorem}
\begin{tabular}{||c||c||c||c||c||c||c||c||c||c||c||c||}
\hline
IC&Der$(\mathcal{I})$ &$Dim(\mathcal{I})$&IC&Der$(\mathcal{I})$&$Dim(\mathcal{I})$\\
			\hline
$\mathcal{TH}_2^{1}$&
$\left(\begin{array}{cccc}
0&0\\
\mathcal{I}_{21}&0
\end{array}
\right)$
&
1
&
$\mathcal{TH}_2^{2}$&
$\left(\begin{array}{cccc}
0&0\\
\mathcal{I}_{21}&0
\end{array}
\right)$
&
1
\\ \hline
$\mathcal{TH}_2^{3}$&
$\left(\begin{array}{cccc}
0&0\\
\mathcal{I}_{21}&0
\end{array}
\right)$
&
1
&
$\mathcal{TH}_2^{4}$&
$\left(\begin{array}{cccc}
0&0\\
\mathcal{I}_{21}&\mathcal{I}_{22}
\end{array}
\right)$
&
1
\\ \hline
$\mathcal{TH}_2^{6}$&
$\left(\begin{array}{cccc}
0&0\\
\mathcal{I}_{21}&0
\end{array}
\right)$
&
1
&
&
&
\\ \hline
\end{tabular}
\begin{proof}
Resting upon Theorem \ref{Ithieo1}, we provide the proof only for one case to illustrate the  used approach. The other cases can handled similarly with or without modification(s).
Let us consider $\mathcal{TH}_2^{4}$. Applying the systems of equations (\ref{Ieq2}), (\ref{Ieq3}), (\ref{Ieq4}) and (\ref{Ieq5}), we get $\mathcal{I}_{11}=\mathcal{I}_{12}=0.$
Thus, the Inner-derivations of $\mathcal{TH}_2^{4}$ are expressed as follows\\
$\mathcal{I}_1=\left(\begin{array}{cc}
0&0\\
1&0
\end{array}
\right)$,\quad $\mathcal{I}_2=
\left(\begin{array}{ccc}
0&0\\
0&1
\end{array}
\right)$ is the basis of $Der(\mathcal{\mathcal{I}})$ and Dim$Der(\mathcal{\mathcal{I}})=2.$ The Inner-derivations of the remaining parts of the two-dimension  associative
trialgebras can be tackled in a similar manner as depicted above.
\end{proof}

\begin{theorem}\label{Ithieo2}
The $\alpha$-\textbf{Inner-derivations} of $3$-dimensional Hom-associative trialgebras have the following form :
\end{theorem}

\begin{tabular}{||c||c||c||c||c||c||c||c||c||c||c||c||}
\hline
IC&Der$(\mathcal{I})$ &$Dim(\mathcal{I})$&IC&Der$(\mathcal{I})$&$Dim(\mathcal{I})$\\
			\hline
$\mathcal{TH}_3^1$&
$\left(\begin{array}{cccc}
\mathcal{I}_{11}&0&0\\
0&\mathcal{I}_{22}&\mathcal{I}_{23}\\
0&-\mathcal{I}_{22}&-\mathcal{I}_{23}
\end{array}
\right)$
&
3
&
$\mathcal{TH}_3^{2}$&
$\left(\begin{array}{cccc}
\mathcal{I}_{11}&0&0\\
0&\mathcal{I}_{22}&\mathcal{I}_{23}\\
0&-\mathcal{I}_{22}&-\mathcal{I}_{23}
\end{array}
\right)$
&
3
\\ \hline
$\mathcal{TH}_3^3$&
$\left(\begin{array}{cccc}
\mathcal{I}_{11}&\mathcal{I}_{21}&0\\
0&0&0\\
0&0&\mathcal{I}_{33}
\end{array}
\right)$
&
3
&
$\mathcal{TH}_3^{4}$&
$\left(\begin{array}{cccc}
\mathcal{I}_{11}&0&\mathcal{I}_{13}\\
0&0&0\\
0&0&0
\end{array}
\right)$
&
2
\\ \hline
$\mathcal{TH}_3^5$&
$\left(\begin{array}{cccc}
0&0&0\\
0&0&\mathcal{I}_{23}\\
0&0&\mathcal{I}_{33}
\end{array}
\right)$
&
2
&
$\mathcal{TH}_3^{6}$&
$\left(\begin{array}{cccc}
\mathcal{I}_{11}&0&\mathcal{I}_{13}\\
0&\mathcal{I}_{22}&0\\
-\mathcal{I}_{11}&0&-\mathcal{I}_{13}
\end{array}
\right)$
&
3
\\ \hline
$\mathcal{TH}_3^7$&
$\left(\begin{array}{cccc}
\mathcal{I}_{11}&\mathcal{I}_{12}&0\\
-\mathcal{I}_{11}&-\mathcal{I}_{12}&0\\
0&0&\mathcal{I}_{33}
\end{array}
\right)$
&
3
&
$\mathcal{TH}_3^{8}$&
$\left(\begin{array}{cccc}
\mathcal{I}_{11}&\mathcal{I}_{12}&0\\
-\mathcal{I}_{11}&-\mathcal{I}_{12}&0\\
0&0&\mathcal{I}_{33}
\end{array}
\right)$
&
3
\\ \hline
\end{tabular}

\begin{tabular}{||c||c||c||c||c||c||c||c||c||c||c||c||}
\hline
IC&Der$(\mathcal{I})$ &$Dim(\mathcal{I})$&IC&Der$(\mathcal{I})$&$Dim(\mathcal{I})$\\
			\hline
$\mathcal{TH}_3^{10}$&
$\left(\begin{array}{cccc}
0&0&0\\
\mathcal{I}_{21}&0&0\\
0&0&\mathcal{I}_{33}
\end{array}
\right)$
&
2
&
$\mathcal{TH}_3^{11}$&
$\left(\begin{array}{cccc}
0&0&0\\
\mathcal{I}_{21}&0&0\\
0&0&\mathcal{I}_{33}
\end{array}
\right)$
&
2
\\ \hline
$\mathcal{TH}_3^{12}$&
$\left(\begin{array}{cccc}
0&0&0\\
\mathcal{I}_{21}&\mathcal{I}_{23}&0\\
0&0&0
\end{array}
\right)$
&
2
&
$\mathcal{TH}_3^{13}$&
$\left(\begin{array}{cccc}
0&0&0\\
\mathcal{I}_{21}&\mathcal{I}_{23}&0\\
0&0&0
\end{array}
\right)$
&
2
\\ \hline
$\mathcal{TH}_3^{14}$&
$\left(\begin{array}{cccc}
0&0&0\\
\mathcal{I}_{21}&\mathcal{I}_{23}&0\\
0&0&0
\end{array}
\right)$
&
3
&
$\mathcal{TH}_3^{15}$&
$\left(\begin{array}{cccc}
0&0&0\\
\mathcal{I}_{21}&\mathcal{I}_{23}&0\\
0&0&0
\end{array}
\right)$
&
2
\\ \hline
$\mathcal{TH}_3^{16}$&
$\left(\begin{array}{cccc}
0&0&0\\
\mathcal{I}_{21}&\mathcal{I}_{23}&0\\
0&0&0
\end{array}
\right)$
&
2
&
$\mathcal{TH}_3^{17}$&
$\left(\begin{array}{cccc}
0&0&0\\
\mathcal{I}_{21}&\mathcal{I}_{23}&0\\
0&0&0
\end{array}
\right)$
&
2
\\ \hline
$\mathcal{TH}_3^{18}$&
$\left(\begin{array}{cccc}
0&0&0\\
\mathcal{I}_{21}&\mathcal{I}_{23}&0\\
0&0&0
\end{array}
\right)$
&
2
&
$\mathcal{TH}_3^{19}$&
$\left(\begin{array}{cccc}
\mathcal{I}_{21}&0&0\\
0&0&0\\
0&0&0
\end{array}
\right)$
&
1
\\ \hline
\end{tabular}

\begin{proof}
Departing from Theorem \ref{Ithieo2}, we provide the proof only for one case to illustrate the used approach, the other cases can be addressed similarly with or without
modification(s). Let's consider ${Trias}_3^{1}$. Applying the systems of equations (\ref{Ieq2}), (\ref{Ieq3}), (\ref{Ieq4}) and (\ref{Ieq5}), we get
$\mathcal{I}_{21}=\mathcal{I}_{31}=\mathcal{I}_{21}=\mathcal{I}_{31}=0$. Hence, the derivations of ${Trias}_3^{1}$ are indicated as follows\\
$\mathcal{I}_1=\left(\begin{array}{cccc}
1&0&0\\
0&0&0\\
0&0&0
\end{array}
\right)$,\quad $\mathcal{I}_2=\left(\begin{array}{cccc}
0&0&0\\
0&1&0\\
0&-1&0
\end{array}
\right)$\quad and \quad$\mathcal{I}_3=\left(\begin{array}{cccc}
0&0&0\\
0&0&1\\
0&0&-1
\end{array}
\right)$
is the basis of $Der(\mathcal{I})$ and Dim$Der(\mathcal{I})=3.$ The centroids of the remaining parts of dimension three associative trialgebras can be handled in a
similar manner as illustrated above.
\end{proof}

\begin{corollary}\,
\begin{itemize}
	\item The dimensions of the $\alpha$-\textbf{Inner-derivations} of \textbf{ $2$}-dimensional Hom-associative trialgebras range between zero and one.
	\item The dimensions of the $\alpha$-\textbf{Inner-derivations} of \textbf{$3$}-dimensional Hom-associative trialgebras range between zero and three.
\end{itemize}
\end{corollary}

\section{\textbf{Centroids} of low-dimensional Hom-associative trialgebras.}
\subsection{Properties of \textbf{centroids} Hom-associative trialgebras.}
In this section, we draw the following results on properties of centroids of Hom-associative trialgebras $\mathcal{A}$.
\begin{definition}
Let $(\mathcal{A}, \dashv, \vdash,  \alpha)$ be a  Hom-associative trialgebra. A linear map
 $\psi : \mathcal{A}\rightarrow \mathcal{A}$ is called an element of $(\alpha)$-\textbf{centroids} on $\mathcal{A}$ if, for all $x, y\in \mathcal{A}$,
\begin{eqnarray}
\alpha\circ\psi&=&\psi\circ\alpha,\\
\psi(x)\dashv \alpha(y)&=&\psi(x)\dashv\psi(y)=\alpha(x)\dashv \psi(y),\\
 \psi(x)\vdash \alpha(y)&=&\psi(x)\vdash \psi(y)=\alpha(x)\vdash \psi(y),\\
 \psi(x)\bot \alpha(y)&=&\psi(x)\bot \psi(y)=\alpha(x)\bot\psi(y).
\end{eqnarray}
The set of all  elements of $(\alpha)$-\textbf{centroid} of $\mathcal{A}$ is denoted $Cent_{(\alpha)}(\mathcal{A})$.
The \textbf{centroid} of $\mathcal{A}$ is denoted $Cent(\mathcal{A})$.
 \end{definition}

\begin{definition}
Let $\mathcal{H}$ be a nonempty subset of $\mathcal{A}$. The subset
\begin{equation}
Z_{\mathcal{A}}(\mathcal{H})=\left\{x\in\mathcal{H} | \alpha(x)\bullet \mathcal{H} = \mathcal{H}\bullet\alpha(x)=0\right\},
\end{equation}
is said to be centralizer of $\mathcal{H}$ in $\mathcal{A}$, where the $\bullet$ is $\dashv$ and $\vdash$, respectively.
\end{definition}

\begin{definition}
Let $\psi\in End(\mathcal{A})$. If $\psi(\mathcal{A})\subseteq Z(\mathcal{A})$ and $\psi(\mathcal{A}^2)=0$, then $\psi$ is called a central derivation.
The set of all central derivations of $\mathcal{A}$ is  denoted by $\mathcal{C}(\mathcal{A})$.
\end{definition}

\begin{proposition}
Consider $(\mathcal{A}, \dashv, \vdash, \alpha)$ a Hom-associative trialgebra. Then,
\begin{enumerate}
	\item [i)]$\Gamma(\mathcal{A})Der(\mathcal{T})\subseteq Der(\mathcal{A})$.
		\item [ii)]$\left[\Gamma(\mathcal{A}), Dr(\mathcal{A})\right]\subseteq\Gamma(\mathcal{A}).$
	\item [iii)]$\left[\Gamma(\mathcal{A}), \Gamma(\mathcal{A})\right](\mathcal{A})\subseteq \Gamma(\mathcal{A})$ and $\left[\Gamma(\mathcal{A}), \Gamma(\mathcal{A})\right](\mathcal{A}^2)=0.$
\end{enumerate}
 \end{proposition}
\begin{proof}
The proof
 of parts $i)-iii)$ is straightforward with reference to definitions of derivation and centroid.
\end{proof}


\begin{proposition}
Let $(\mathcal{A}, \dashv, \vdash, \bot, \alpha)$ be a Hom-associative trialgebra and $\varphi\in Cent(\mathcal{A}),\, d\in Der(\mathcal{A}).$
Then, $\varphi\circ d$ is an $\alpha$-derivation of $\mathcal{A}.$
\end{proposition}
\begin{proof}
Indeed, if $x, y\in \mathcal{A}$, then
$$\begin{array}{ll}
(\varphi\circ d)(x\bullet y)
&= \varphi(d(x)\bullet\alpha(y)+\alpha(x)\bullet d(y))\\
&= \varphi(d(x)\bullet y)+\varphi(x\bullet d(y))=(\varphi\circ d)(x)\bullet\alpha(y)+\alpha(x)\bullet(\varphi\circ d)(y),
\end{array}$$
where $\bullet$ is $\dashv, \vdash$ and $\bot$, respectively.
\end{proof}

\begin{proposition}
Let $(\mathcal{A}, \dashv, \vdash, \bot, \alpha)$ be a Hom-associative trialgebra over a field $\mathbb{F}$. Hence, $\mathcal{C}(\mathcal{A})=Cent(\mathcal{A})\cap Der(\mathcal{A}).$
\end{proposition}

\begin{proof}
If $\psi\in Cent(\mathcal{A})\cap Der(\mathcal{A})$, then by definition of $Cent(\mathcal{A}$ and $Der(\mathcal{A})$, we get

$\psi(x\bullet y)=\psi(x)\bullet\alpha(y)+\alpha(x)\bullet\psi(y)$ and $\psi(x\bullet y)=\psi(x)\circ\alpha(y)=\alpha(x)\circ\psi(y)$ for $x,y\in \mathcal{A}.$
The yields $\psi(\mathcal{A}\mathcal{A})=0$ and $\psi(\mathcal{A})\subseteq  Z(\mathcal{A})$, i.e, $Cent(\mathcal{A})\cap Der(\mathcal{A})\subseteq Cent(\mathcal{A}).$
The opposite is obvious since $\mathcal{C}(\mathcal{A})$ is in both $Cent(\mathcal{A})$ and $Der(\mathcal{A}),$ where $\bullet$ is $\dashv, \vdash$ and $\bot$, respectively.
\end{proof}

\begin{proposition}
Let $(\mathcal{A}, \dashv, \vdash, \bot, \alpha)$ be a Hom-associative trialgebra. Therefore,  for any $d\in Der(\mathcal{A})$ and $\varphi\in Cent(\mathcal{A})$, we have
\begin{enumerate}
	\item [(i)] The composition $d\circ\varphi$ is in $Cent(\mathcal{A})$, if and only if $\varphi\circ d$ is a central $\alpha$-derivation of $\mathcal{A}.$
		\item [(ii)] The  composition $d\circ\varphi$ is a $\alpha$-derivation of $\mathcal{A}$, if and only if $\left[d,\varphi\right]$ is a central $\alpha$-derivation of $\mathcal{A}.$
\end{enumerate}
\end{proposition}

\begin{proof}
\begin{enumerate}
	\item [i)]For any $\varphi\in Cent(\mathcal{A}),\, d\in Der(\mathcal{A}),\, \forall\,x,y\in \mathcal{A}$, we have
	$$\begin{array}{ll}
d\circ\varphi(x\bullet y)=d\circ\varphi(x)\bullet y
&=d\circ\varphi(x)\bullet y+\varphi(x)\bullet d(y)\\
&=d\circ\varphi(x)\bullet y+\varphi\circ d(x\bullet y)-\varphi\circ d(x)\bullet y.
\end{array}$$
Thus, $(d\circ\varphi-\varphi\circ d)(x\bullet y)=(d\bullet\varphi-\varphi\circ d)(x)\bullet y.$
	\item [ii)] Let $d\circ\varphi\in Der(\mathcal{A})$. Using $\left[d,\varphi\right]\in Cent(\mathcal{A})$, we get
	\begin{equation}\label{eq1}
	\left[d,\varphi\right](x\bullet y)=(\left[d, \varphi\right](x))\bullet\alpha(y)=\alpha(x)\bullet(\left[d,\varphi\right](y))
	\end{equation}
	On the other side, $\left[d, \varphi\right]d\circ\varphi-\varphi\circ d$ and $d\circ\varphi, \varphi\circ d\in Der(\mathcal{A}).$ Therefore,	
	\begin{equation}\label{eq2}
\left[d, \varphi\right](x\bullet y)=(d(\varphi\circ(x))\bullet\alpha(y)+\alpha(x)\bullet(d\circ\varphi(y))-(\varphi\circ d(x))\bullet\alpha(y)-\alpha(x)\bullet(\varphi\circ d(y)).
\end{equation}
Relying upon (\ref{eq1}) and (\ref{eq2}), we get $\alpha(x)\bullet(\left[d, \varphi\right])(y)=(\left[d, \varphi\right])(x)\bullet\alpha(y)=0.$

\noindent At this stage of analysis,let $\left[d, \varphi\right]$ be a central $\alpha$-derivation of $\mathcal{A}$. Then,
$$\begin{array}{ll}
d\circ\varphi(x\bullet y)
&=\left[d\circ\varphi\right](x\bullet y)+(\varphi\circ d)(x\bullet y)\\
&=\varphi(\circ d(x)\bullet\alpha(y))+\varphi(\alpha(x)\bullet d(y))\\
&=(\varphi\circ d)(x)\bullet\alpha(y)+\alpha(x)\bullet(\varphi\circ d)(y),
\end{array}$$
\end{enumerate}
where $\bullet$ indicates the products $\dashv, \vdash$ and $\bot$, respectively.
\end{proof}

This section provides pertinent  details on $\alpha$-derivation of Hom-associative trialgebras in dimension two and three over the field $\mathbb{K}.$ Let
$\left\{e_1,e_2, e_3,\cdots, e_n\right\}$ be a basis of an $n$-dimensional Hom-associative trialgebra $\mathcal{A}.$ The product of basis
\begin{eqnarray}
\psi(e_p)=\sum_{q=1}^nc_{qp}e_q\nonumber.
\end{eqnarray}

\begin{eqnarray}
\sum_{p=1}^nc_{pi}a_{qp}&=&\sum_{p=1}^na_{pi}c_{qp},\label{Ceq2}\\
\sum_{k=1}^n\gamma_{ij}^kc_{qk}=\sum_{k=1}^n\sum_{p=1}^nc_{ki}a_{pj}\gamma_{kp}^q&;&\sum_{k=1}^n\gamma_{ij}^kc_{qk}=\sum_{k=1}^n\sum_{p=1}^na_{ki}c_{pj}\gamma_{kp}^q,\label{Ceq3}\\
\sum_{k=1}^n\delta_{ij}^kc_{qk}=\sum_{k=1}^n\sum_{p=1}^nc_{ki}a_{pj}\delta_{kp}^q&;&\sum_{k=1}^n\delta_{ij}^kc_{qk}=\sum_{k=1}^n\sum_{p=1}^na_{ki}c_{pj}\delta_{kp}^q,\label{Ceq4}\\
\sum_{k=1}^n\phi_{ij}^kc_{qk}=\sum_{k=1}^n\sum_{p=1}^nc_{ki}a_{pj}\phi_{kp}^q&;&\sum_{k=1}^n\phi_{ij}^kc_{qk}=\sum_{k=1}^n\sum_{p=1}^na_{ki}c_{pj}\phi_{kp}^q\label{Ceq5}.
\end{eqnarray}

\begin{theorem}\label{Cthieo1}
The \textbf{centroids} of $2$-dimensional complex Hom-associative trialgebras are depicted as follows :
\end{theorem}

\begin{tabular}{||c||c||c||c||c||c||c||c||c||c||c||c||}
\hline
IC&$Cent(\mathcal{A})$ &$Dim(Cent(\mathcal{A}))$&IC&$Cent(\mathcal{A})$&$Dim(Cent(\mathcal{A}))$\\
			\hline
$\mathcal{TH}_2^{1}$&
$\left(\begin{array}{cccc}
0&0\\
c_{21}&0
\end{array}
\right)$
&
1
&
$\mathcal{TH}_2^{2}$&
$\left(\begin{array}{cccc}
0&0\\
c_{21}&0
\end{array}
\right)$
&
1
\\ \hline
$\mathcal{TH}_2^{4}$&
$\left(\begin{array}{cccc}
c_{11}&0\\
c_{21}&d_{11}
\end{array}
\right)$
&
2
&
$\mathcal{TH}_2^{5}$&
$\left(\begin{array}{cccc}
c_{11}&0\\
c_{21}&c_{11}
\end{array}
\right)$
&
2
\\ \hline
$\mathcal{TH}_2^{6}$&
$\left(\begin{array}{cccc}
c_{11}&0\\
c_{21}&c_{11}
\end{array}
\right)$
&
2
&
$\mathcal{TH}_2^{7}$&
$\left(\begin{array}{cccc}
c_{11}&0\\
0&c_{11}
\end{array}
\right)$
&
1
\\ \hline
$\mathcal{TH}_2^{8}$&
$\left(\begin{array}{cccc}
c_{11}&0\\
0&c_{11}
\end{array}
\right)$
&
1
&
$\mathcal{TH}_2^{10}$&
$\left(\begin{array}{cccc}
c_{11}&0\\
0&0
\end{array}
\right)$
&
1
\\ \hline
$\mathcal{TH}_2^{11}$&
$\left(\begin{array}{cccc}
0&0\\
0&c_{22}
\end{array}
\right)$
&
1
&
$\mathcal{TH}_2^{12}$&
$\left(\begin{array}{cccc}
c_{11}&0\\
0&0
\end{array}
\right)$
&
1
\\ \hline
$\mathcal{TH}_2^{13}$&
$\left(\begin{array}{cccc}
c_{11}&0\\
0&c_{11}
\end{array}
\right)$
&
1
&
&
&
\\ \hline
\end{tabular}

\begin{proof}
Departing from Theorem \ref{Cthieo1}, we provide the proof only for one case to illustrate the used approach, the other cases can be addressed similarly with or without
modification(s). Let's consider $\mathcal{TH}_2^{4}$. Applying the systems of equations (\ref{Ceq2}), (\ref{Ceq3}), (\ref{Ceq4}) and (\ref{Ceq5}), we get
$c_{12}=0$ and $c_{22}=c_{11}$. Hence, the centroids of $\mathcal{TH}_2^{4}$ are indicated as follows\\
$c_1=\left(\begin{array}{cccc}
1&0\\
0&1
\end{array}
\right)$,\quad $c_2=\left(\begin{array}{cccc}
0&0\\
1&0
\end{array}
\right)$
is the basis of $Der(c)$ and  Dim$Der(c)=2.$ The centroids of the remaining parts of dimension three associative trialgebras can be handled in a
similar manner as illustrated above.
\end{proof}

\begin{theorem}\label{Cthieo2}
The \textbf{centroids} of $3$-dimensional complex Hom-associative trialgebras are depicted as follows :
\end{theorem}

\begin{tabular}{||c||c||c||c||c||c||c||c||c||c||c||c||}
\hline
IC&$Cent(\mathcal{A})$ &$Dim(Cent(\mathcal{A}))$&IC&$Cent(\mathcal{A})$&$Dim(Cent(\mathcal{A}))$\\
			\hline
$\mathcal{TH}_3^1$&
$\left(\begin{array}{cccc}
c_{11}&0&0\\
0&c_{22}&c_{23}\\
0&-c_{22}&-c_{23}
\end{array}
\right)$
&
3
&
$\mathcal{TH}_3^{2}$&
$\left(\begin{array}{cccc}
c_{11}&0&0\\
0&c_{22}&c_{23}\\
0&-c_{22}&-c_{23}
\end{array}
\right)$
&
3
\\ \hline
$\mathcal{TH}_3^3$&
$\left(\begin{array}{cccc}
0&c_{12}&0\\
0&0&0\\
0&0&0
\end{array}
\right)$
&
1
&
$\mathcal{TH}_3^{4}$&
$\left(\begin{array}{cccc}
c_{11}&0&c_{13}\\
0&0&0\\
0&0&c_{11}
\end{array}
\right)$
&
2
\\ \hline
$\mathcal{TH}_3^5$&
$\left(\begin{array}{cccc}
0&0&0\\
0&c_{22}&\\
0&0&c_{33}
\end{array}
\right)$
&
2
&
$\mathcal{TH}_3^{6}$&
$\left(\begin{array}{cccc}
c_{11}&0&c_{13}\\
0&c_{22}&0\\
-c_{11}&0&-c_{13}
\end{array}
\right)$
&
3
\\ \hline
$\mathcal{TH}_3^7$&
$\left(\begin{array}{cccc}
c_{11}&c_{12}&0\\
-c_{11}&-c_{12}&0\\
0&0&c_{33}
\end{array}
\right)$
&
3
&
$\mathcal{TH}_3^{8}$&
$\left(\begin{array}{cccc}
c_{11}&c_{12}&0\\
-c_{11}&-c_{12}&0\\
0&0&c_{33}
\end{array}
\right)$
&
3
\\ \hline
$\mathcal{TH}_3^{10}$&
$\left(\begin{array}{cccc}
0&0&0\\
c_{21}&0&0\\
0&0&c_{33}
\end{array}
\right)$
&
2
&
$\mathcal{TH}_3^{11}$&
$\left(\begin{array}{cccc}
0&0&0\\
c_{21}&0&0\\
0&0&c_{33}
\end{array}
\right)$
&
2
\\ \hline
$\mathcal{TH}_3^{12}$&
$\left(\begin{array}{cccc}
0&0&0\\
c_{21}&c_{23}&0\\
0&0&0
\end{array}
\right)$
&
2
&
$\mathcal{TH}_3^{13}$&
$\left(\begin{array}{cccc}
0&0&0\\
c_{21}&c_{23}&0\\
0&0&0
\end{array}
\right)$
&
2
\\ \hline
$\mathcal{TH}_3^{14}$&
$\left(\begin{array}{cccc}
c_{11}&0&0\\
c_{21}&c_{11}&c_{23}\\
-c_{11}&0&0
\end{array}
\right)$
&
3
&
$\mathcal{TH}_3^{15}$&
$\left(\begin{array}{cccc}
0&0&0\\
c_{21}&0&c_{23}\\
0&0&0
\end{array}
\right)$
&
2
\\ \hline
$\mathcal{TH}_3^{16}$&
$\left(\begin{array}{cccc}
0&0&0\\
c_{21}&0&c_{23}\\
0&0&0
\end{array}
\right)$
&
2
&
$\mathcal{TH}_3^{17}$&
$\left(\begin{array}{cccc}
0&0&0\\
c_{21}&0&c_{23}\\
0&0&0
\end{array}
\right)$
&
2
\\ \hline
$\mathcal{TH}_3^{18}$&
$\left(\begin{array}{cccc}
0&0&0\\
c_{21}&0&c_{23}\\
0&0&0
\end{array}
\right)$
&
2
&
$\mathcal{TH}_3^{20}$&
$\left(\begin{array}{cccc}
0&0&0\\
c_{21}&0&0\\
c_{31}&c_{21}&0
\end{array}
\right)$
&
2
\\ \hline
$\mathcal{TH}_3^{21}$&
$\left(\begin{array}{cccc}
0&0&0\\
c_{21}&0&0\\
c_{31}&c_{21}&0
\end{array}
\right)$
&
2
&
&
&
\\ \hline
\end{tabular}

\begin{proof}
Departing from Theorem \ref{Cthieo2}, we provide the proof only for one case to illustrate the used approach, the other cases can be addressed similarly with or without
modification(s). Let's consider $\mathcal{TH}_3^{6}$. Applying the systems of equations (\ref{Ceq2}), (\ref{Ceq3}), (\ref{Ceq4}) and (\ref{Ceq5}), we get
$c_{12}=c_{21}=c_{23}=c_{32}=0$,\quad $c_{11}=-c_{31}$ and $c_{13}=-c_{33}$. Hence, the centroids of $\mathcal{TH}_3^{6}$ are indicated as follows\\
$c_1=\left(\begin{array}{cccc}
1&0&0\\
0&0&0\\
-1&0&0
\end{array}
\right)$,\quad $c_2=\left(\begin{array}{cccc}
0&0&0\\
0&1&0\\
0&0&0
\end{array}
\right)$\quad and \quad$c_3=\left(\begin{array}{cccc}
0&0&0\\
0&0&1\\
0&0&-1
\end{array}
\right)\quad$
is the basis of $Der(c)$ and\\ Dim$Der(c)=3.$ The centroids of the remaining parts of dimension three associative trialgebras can be handled in a
similar manner as illustrated above.
\end{proof}

\begin{corollary}\,
\begin{itemize}
	\item The dimensions of the \textbf{centroids} of two-dimensional Hom-associative trialgebras range between zero and two.
	\item The dimensions of the \textbf{centroids} of three-dimensional Hom-associative trialgebras range between zero and three.
\end{itemize}
\end{corollary}

\end{document}